\documentclass[11pt, twoside]{amsart}

\usepackage{amsmath}
\usepackage{amsfonts}
\usepackage{amssymb}
\usepackage{amsthm}
\usepackage{bbm}
\usepackage{verbatim} 

\usepackage[T1]{fontenc}

\theoremstyle{plain}
\newtheorem{thm}{Theorem}[section]
\newtheorem{cor}[thm]{Corollary}
\newtheorem{lma}[thm]{Lemma}
\newtheorem{prp}[thm]{Proposition}

\newtheorem{fact}[thm]{Fact}

\theoremstyle{definition}
\newtheorem{defi}[thm]{Definition}
\newtheorem{rmk}[thm]{Remark}
\newtheorem{exa}[thm]{Example}

\theoremstyle{remark}

\numberwithin{equation}{section}

\newcommand{\cl}{\mathrm{cl}}

\newcommand{\im}{\mathrm{im}}

\newcommand{\reduct}{\upharpoonright}

\newcommand{\mcA}{\mathcal{A}}
\newcommand{\mcB}{\mathcal{B}}

\newcommand{\mcG}{\mathcal{G}}
\newcommand{\mcH}{\mathcal{H}}

\newcommand{\mcM}{\mathcal{M}}
\newcommand{\mcN}{\mathcal{N}}

\newcommand{\mcP}{\mathcal{P}}

\newcommand{\mbC}{\mathbf{C}}

\newcommand{\mbK}{\mathbf{K}}
\newcommand{\mbS}{\mathbf{S}}

\newcommand{\mbbN}{\mathbb{N}}

\newcommand{\mbbZ}{\mathbb{Z}}

\newcommand{\tb}{\textbf}
\newcommand{\meq}{{\mathcal{M}^{eq}}}

\newcommand{\ind}{\raisebox{-2pt}[5pt][0pt]{$\smile$} \hspace*{-6.8pt}\raisebox{3pt}[5pt][0pt]{$|$} \; \: }
\newcommand{\nind}{\raisebox{-2pt}[5pt][0pt]{$\smile$} 
\hspace*{-6.8pt}\raisebox{3pt}[5pt][0pt]{$|$}\hspace*{-6.8pt}
\raisebox{3pt}[5pt][0pt]{$\diagup$} }

\title{Simple structures axiomatized by almost sure theories}
\date{}
\author{Ove Ahlman}
\address{Ove Ahlman, Department of Mathematics, Uppsala University, Box 480, 75 106 Uppsala, Sweden}
\email{ove@math.uu.se}

\begin{document}
\maketitle
\begin{abstract}
In this article we give a classification of the binary, simple, $\omega-$categorical structures with $SU-$rank $1$ and trivial pregeometry. This is done both by showing that they satisfy certain extension properties, but also by noting that they may be approximated by the almost sure theory of some sets of finite structures equipped with a probability measure. This study give results about general almost sure theories, but also considers certain attributes which, if they are almost surely true, generate almost sure theories with very specific properties such as $\omega-$stability or strong minimality.
\end{abstract}
\section{Introduction}\label{introsection}
\noindent For each $n\in\mbbN$, let $\mbK_n$ be a non-empty finite set of finite structures equipped with a probability measure $\mu_n$ and let $\mbK  =(\mbK_n,\mu_n)_{n\in\mbbN}$. For any property $\textbf{P}$ (often a sentences in the language) we may extend the measure $\mu_n$ to associate a probability with $\textbf{P}$ by defining
\[\mu_n(\textbf{P}) = \mu_n\{\mcM\in\mbK : \mcM \text{ satisfies } \tb{P}\}.\]
A property \tb{P} such that $\lim_{n\rightarrow\infty} \mu_n(\tb{P})=1$ is said to be an \tb{almost sure} property of (also called 'almost surely true in') $\mbK$.
The set $\mbK$ is said to have a \textbf{$0-1$ law} if, for each sentence $\varphi$ in the language, either $\varphi$ or $\neg\varphi$ is almost sure in $\mbK$ i.e. each formula has asymptotic probability $0$ or $1$. The almost sure theory of $\mbK$, denoted $T_\mbK$, is the set of all almost sure sentences. Notice that $\mbK$ has a $0-1$ law if and only if $T_\mbK$ is complete. A theory is called $\omega-$\textbf{categorical} if it has a unique countable model up to isomorphism. The following fact leads us to see that one method to show that a set $\mbK$ has a $0-1$ law is to prove that $T_\mbK$ is $\omega-$categorical.
\begin{fact}Let $T$ be a theory which is categorical in some infinite cardinality. Then $T$ has no finite models if and only if $T$ is complete.
\end{fact}
Many $0-1$ laws \cite{AK2, C, F, KPR, K} are proved in this way and additionally the corresponding almost sure theories are supersimple with $SU-$rank 1 and have trivial pregeometry. In this article we ask ourself what the reason is for this pattern and whether the supersimple $\omega-$categorical theories with SU-rank $1$ tend to be almost sure theories. In \cite{AK3} the author together with Koponen studied sets of $SU-$rank 1 in homogeneous simple structures with a binary vocabulary. In this case, a strong connection was found to both trivial pregeometries but also to random structures and almost sure theories. The present article will explore these implications further and prove the following theorem.
\begin{thm}\label{rngstrthm}
If $T$ is $\omega-$categorical, simple with SU-rank 1 and trivial pregeometry over a binary relational vocabulary then there exists a set of finite structures $\mbK= (\mbK_n,\mu_n)_{n\in\mbbN}$ with a probability measure $\mu_n$ such that $T_\mbK = T$.
\end{thm}
The key to $\omega-$categorical almost sure theories is the notion of extension properties. These are first order sentences which state that if we have a tuple of a certain atomic diagram, possibly satisfying certain extra properties, then we may extend this into a larger tuple which also satisfies certain specific first order formulas. The connection to $\omega-$categorical theories is very clear as these properties describe how we stepwise should build an isomorphism, and the method has been used before to prove many previous $0-1$ laws, \cite{C, EF, F,K, MT} among others. It is possible to make the extension properties very specific and in this way we will get a characterization of the simple $\omega-$categorical structures with $SU-$rank $1$ with trivial pregeometry by stating how their extension properties should look like. Furthermore the way the extension properties are created implies that these structures are not only $\omega-$categorical but also homogenizable i.e. we may add a relational symbol to an $\emptyset-$definable relation to make it a homogeneous structure.

Studies of general almost sure theories have been done before, the most common is the extension of the Er\H{o}s-R\'{e}nyi random graph which have been studied by Baldwin \cite{B} among others. This construction is though quite different from what we apply in this article, which is especially clear since their almost sure theories are not in general $\omega-$categorical, though stable, as Baldwin points out.\\\indent 
A definable pregeometry is an especially interesting part in almost sure theories and simple theories. The $0-1$ laws just mentioned all have trivial pregeometries. However the author together with Koponen constructed in \cite{AK} a $0-1$ law for structures which almost surely define a vector space pregeometry such that the structures also are restricted by certain colouring axioms. The question arises, why so few non-trivial pregeometries are found in almost sure theories, and we will in this article partially answer it by two different results. One answer is that if we have simple enough extension properties (the common method to show $0-1$ laws) then we will almost surely have a trivial pregeometry in the almost sure theory. The other answer is that if the sets of structures $\mbK_n$ are such that $|\mcN|=n$ for all $\mcN\in \mbK_n$ then vector space pregeometries (or affine or projective geometries) do almost surely not exist. 
The pregeometries of simple structures are often vector space like (or affine/projective), and thus if we want to create simple $\omega-$categorical almost sure theories with nontrivial pregeometry, we may conclude from results in this article that we will need classes of structures which grow non-linearly together with more interesting extension properties.

There are many different ways to construct infinite structures from finite structures or by probabilistic methods. Fra\"{i}ss\'{e} \cite{Fr} showed that having a set of finite structures satisfying certain properties, among them amalgamation, generates a unique infinite homogeneous structure. This infinite structure contain the initial set of finite structures as substructures, and thus inherits many properties from these. In \cite{AFP} Ackerman, Freer and Patel discusses which countable structures are approximable by a probability measure on all possible infinite structures with a fixed countably infinite universe. They discover that this is equivalent with having a trivial definable closure, which is an interesting contrast to the results in this article where we notice that in order to approximate infinite structures by using finite structures and a probability measure is closely related to how the algebraic closure work, and is easiest if the algebraic closure is trivial. A classical result by Erd\H{o}s, Kleitman and Rothschild \cite{EKR} shows that if $\mbK_n$ is the set of all triangle free graphs of size $n$ then $(\mbK_n,\mu_n)_{n\in\mbbN}$ is almost surely bipartite under the uniform measure $\mu_n$. However the Fra\"{i}ss\'{e} construction from the set $\bigcup_{n=1}^\infty \mbK_n$ creates a homogeneous structure which is not bipartite. Even more generally, Koponen \cite{K} shows that for certain structures $\mcH$ if $\mbK_n(\mcH)$ is the set of all structures with universe $[n]$ where $\mcH$ is not a weak substructure, then the Fra\"{i}ss\'{e} limit and the almost sure theory of these sets are not the same. In this article we study the question which structures are possible to approximate probabilistically by taking the set of embeddable finite structures with a certain universe under the uniform measure. We call structures which are approximable this way random structures (Definition \ref{rngstrdef} for details) and prove that the following theorem hold.
\begin{thm}\label{reductprp} If $\mcM$ is countable $\omega-$categorical, simple with $SU-$rank 1 and trivial pregeometry such that $acl(\emptyset)=\emptyset$ then $\mcM$ is a reduct of a random structure which is also $\omega-$categorical, simple, with $SU-$rank 1 and has trivial pregeometry.
\end{thm}
In Section \ref{eqsection} we discuss general sets of finite structures $\mbK_n$ with an associated probability measure $\mu_n$ such that $\mbK = (\mbK_n,\mu_n)_{n\in\mbbN}$ has a $0-1$ law and what consequences this has on the almost sure theory $T_\mbK$. We gather important results for the later sections which imply that equivalence relations, and especially pregeometries, which are almost surely definable give a direct restriction on which sizes of structures may exist in $\bigcup_{n\in\mbbN} \mbK_n$.

Section \ref{omegasection} is giving a different approach to $\omega-$categorical theories than what is usually practiced. We introduce the concept of extension properties, commonly used to prove $0-1$ laws, and show that their existence is equivalent with $\omega-$categoricity, which will be useful in later sections. Furthermore we extend the concept of $\mcM^{eq}$ to finite structures, and use this in order to show that for any $0<n<\omega$ there exist almost sure theories which are $\omega-$categorical and simple with $SU-$rank $n$. 

In Section \ref{smsection} we study binary, $\omega-$categorical theories with $SU-$rank $1$ and trivial pregeometry which are either simple, $\omega-$stable or strongly minimal. It turns out that for these theories the concept of extension properties from Section \ref{omegasection} becomes very explicit. Namely there is an equivalence relation $\xi$ such that how a tuple may be extended only depends on which equivalence classes its parts are in. When the theory is $\omega-$stable this extension is unique (to each equivalence class) and when the theory is strongly minimal the extensions are unique and we have only one infinite equivalence class. Added together the extension properties give us an explicit axiomatization for these theories.

In Section \ref{classsection} we combine the results from previous sections in order to study how to approximate an infinite structure with finite ones using probabilistic methods. Specifically we see that structures with certain properties have the same theory as the almost sure theory of a set of finite structures. The main result is Theorem \ref{rngstrthm}, which gives a new $0-1$ law for a set of structures partitioned into a finite amount of equivalence classes, which all have random relations between and inside the classes.
We also make an exposition of the so called random structures and prove Theorem \ref{reductprp}. Moreover we show through examples that the random structures are in general quite hard to pin down, and their existence is more of a combinatorial property than model theoretic.

\section{Preliminaries}\label{prelimsection}
\noindent Following is a brief introduction to the basic concepts used in this article, for a more detailed exposition the reader could study \cite{EF,H}.
A \tb{finite relational vocabulary} is a finite set consisting of relation symbols of certain finite arities. This will be the only kind of vocabulary considered in this article. The vocabulary is binary if it only contain relational symbols of arity at most $2$. A \tb{theory} (or $V-$theory if $V$ denotes the vocabulary) T is a set of sentences created from the vocabulary. If $\mcM$ is a structure over the vocabulary $V$ then the theory of $\mcM$, denoted $Th(\mcM)$, consists of all sentences, from the vocabulary, which are true in $\mcM$. The $V-$structures (or just structures if the vocabulary is obvious) will be written in calligraphic letters like $\mcA,\mcB, \mcM,\mcN,\ldots$ with universes denoted by the corresponding normal letter $A, B, M, N,\ldots$ while ordered tuples of elements (or variables) will be denoted with small letters with bars $\bar a, \bar b,\bar c, \bar x,\bar y,\ldots$ and we will at times identify the tuples with the set of their elements in which case this is made obvious by using set theoretic operations on the tuple. When we write $\bar a\in M$ we mean that $\bar a$ is an ordered tuple consisting only of elements in $M$. The atomic diagram of a tuple $\bar a$ in $\mcM$ will be denoted by $atDiag^\mcM(\bar a)$. We will at times, for a positive integer $n$, use the abbreviation $[n]=\{1,\ldots,n\}$. The cardinality of a set $X$ is denoted $|X|$. 
For any structure $\mcM$, formula $\varphi(\bar x,\bar y)$ and tuple $\bar a\in M$ let
\[\varphi(\mcM,\bar a) = \{\bar b\in M : \mcM\models \varphi(\bar b,\bar a)\}.\]
For any sets $X,Y\subseteq M$, $Y$ is called $X-$\tb{definable} if there exists a formula $\varphi(\bar x,\bar y)$ and tuple $\bar a\in X$ such that $Y=\varphi(\mcM,\bar a)$.\\
\indent 
A \textbf{pregeometry} $(A,cl)$ is a set $A$ together with a set function $cl:\mcP(A)\rightarrow \mcP(A)$ which satisfies the following, for each $X,Y\subseteq A:$
\begin{description}
\item[Reflexivity] $X\subseteq cl(X)$.
\item[Monotonicity] $Y\subseteq cl(X) \Rightarrow cl(Y) \subseteq cl(X)$.
\item[Finite character] $cl(X) = \bigcup \{cl(X_0) : X_0\subseteq X \text{ and } |X_0|< \omega\}$.
\item[Exchange] For each $a,b\in A$, $a\in cl(X\cup \{b\})-cl(X) \Rightarrow b\in cl(X\cup \{a\})$.
\end{description}
We will make notation easier and instead of writing $cl(\{a_1,\ldots,a_n\})$ we may exclude the set brackets and write $cl(a_1,\ldots,a_n)$. A set $\{a_1,\ldots,a_n\}$ of elements in a pregeometry $(A,cl)$ is called \textbf{independent} if for each $i\in [n]$ $a_i\notin cl(\{a_1,\ldots,a_n\}-\{a_i\})$. A pregeometry is called \textbf{trivial} if for each $X\subseteq A$, $cl(X) = X\cup cl(\emptyset)$.
In a structure $\mcM$ with $X\subseteq M$ we say that a pregeometry $(M,cl)$ is $X$-\textbf{definable} if there are formulas $\theta_0(x_0),\ldots,\theta_n(x_0,x_1,\ldots,x_n),\ldots$, possibly using parameters from $X$, such that for each $i\in\mbbN$ and $a_0,\ldots,a_i\in M$, $\mcM\models\theta_i(a_0,a_1,\ldots,a_i)$ if and only if $a_0\in cl(a_1,\ldots,a_i)$. 
We have the following well known fact about $\omega-$categorical theories, which will be used throughout the paper without special mentioning.
\begin{fact}(\textit{Ryll-Nardzewski theorem}) Let $T$ be a theory, then the following are equivalent.
\begin{itemize}
\item $T$ is $\omega-$categorical
\item For each $n<\omega$ there are only a finite amount of $n-$types in $T$.
\item For each $n<\omega$ all $n-$types are isolated.
\end{itemize}
\end{fact}
For each structure $\mcM$ and $X\subseteq M$ let $acl(X)$ denote the algebraic closure of $X$ i.e. all elements whose type over $X$ is algebraic. A theory $T$ is called \textbf{strongly minimal} if for each $\mcM\models T$, formula $\varphi(x,\bar y)$ and tuple $\bar a\in M$ either $\varphi(\mcM,\bar a)$ is finite or $\neg\varphi(\mcM,\bar a)$ is finite. 
We say that $T$ is $\omega-$\textbf{stable} if for each $\mcM\models T$ and $A\subseteq M$ with $|A|=\aleph_0$, the number of different types over $A$ is $\aleph_0$. A structure is called strongly minimal or $\omega-$stable if its theory $Th(\mcM)$ is. The definitions of SU-rank and of a theory being simple/supersimple are extensive and we refer the reader to \cite{W}.
\begin{fact}\label{simppregfact}
For any simple structure $\mcM$ with $SU-$rank $1$, the pair $(M,acl)$ forms a pregeometry.
\end{fact}
\noindent A quick corollary to this, using the Ryll-Nardzewski theorem, is that a pregeometry $(M, acl)$ is $\emptyset-$definable in all such structures $\mcM$ which are $\omega-$categorical. Further notice that all strongly minimal or $\omega-$stable theories are simple and hence we may apply the previous fact to these theories.

\begin{defi} Let $\mcM$ be a $V-$structure with $T= Th(\mcM)$. For each $V-$formula $\varphi(\bar x,\bar y)$ with arity $2e_\varphi$ such that $T$ implies that $\varphi$ defines an equivalence relation let $R_\varphi$ and $P_\varphi$ be new relational symbols of arity $e_\varphi+1$ and $1$ respectively. Put
\[V^{eq}= V\cup \{R_\varphi,P_\varphi : \text{$T$ implies that } \varphi(\bar x,\bar y) \in L \text{ is an equivalence relation}\}.\]
Create the $V^{eq}$ structure $\mcM^{eq}$ by taking $\mcM$ and adding, for each $\emptyset-$definable equivalence relation $\varphi(\bar x,\bar y)$, extra elements which represent the equivalence classes of $\varphi$. Let $\mcM^{eq}\models P_{\varphi}(b)$ iff $b$ represents a $\varphi-$equivalence class, and $\bar a\in M$ is in the $\varphi$ equivalence class $b$ iff $\mcM^{eq}\models R_\varphi(\bar a, b)$. Let $P_{=}$ be a special case such that $P_{=}(\mcM^{eq}) = M$ and call these elements the \textbf{home sort}. The elements in the home sort hence have the same relations as in $\mcM$ when restricting to $V$ and we will assume that no other relations from $\mcM$ are true in $\mcM^{eq}$. The elements in the home sorts are called \textbf{real} and the ones outside are called \textbf{imaginary}. 
\end{defi}
When a structure $\mcM$ is $\omega-$categorical we especially get isolated types over $\emptyset$ by the Ryll-Nardzewski theorem. These properties also hold, partially, in $\mcM^{eq}$ giving us the following fact, which follows from the regular theorems about formulas transferring between $\mcM$ and $\mcM^{eq}$.
\begin{fact}\label{eqfact}
If $\mcM$ is an $\omega-$categorical structure then the equivalence relation $tp_{\mcM^{eq}}(x/acl_{\mcM^{eq}}(\emptyset)) = tp_{\mcM^{eq}}(x/acl_{\mcM^{eq}}(\emptyset)) $ restricted to the home sort is $\emptyset-$definable in $\mcM$.
\end{fact}
In $\omega-$categorical simple theories the concept of 'Lascar strong types' is equivalent with the concept of strong types (\cite{W}, Corollary 6.1.11). Using this fact we may formulate the very useful ``independence theorem'' for simple theories in the following way, adapted for our purposes.

\begin{fact}\label{indthm}
Let $\mcM$ be simple and $\omega-$categorical with $b_1,\ldots,b_n\in M$. Assume $a_1,\ldots, a_n \in M$ and for each $i\neq j$, $b_i\ind b_j$, $a_i \ind b_i$ and 
\[tp_{\meq}(a_i/acl_\meq(\emptyset)) = tp_\meq (a_j /acl_\meq (\emptyset)).\]
Then there exists $c\in M$ such that for each $i\in[n]$
\[tp_{\meq}(c/\{b_i\}\cup acl_\meq(\emptyset)) = tp_{\meq}(a_i/\{b_i\}\cup acl_{\meq}(\emptyset))\]
\end{fact}
\noindent More information about $\mcM^{eq}$ and other model theoretic properties may be found in \cite{H}.

\section{Sets of structures with a 0-1 law}\label{eqsection}
\noindent In this section we will assume that for each $n\in \mbbN, \mbK_n$ is a set of finite structures, $\mu_n$ is a probability measure on $\mbK_n$ and put $\mbK = (\mbK_n, \mu_n)_{n\in\mbbN}$. We assume that almost surely the size of structures in $\mbK$ grow. No further properties are assumed, such as labeled/unlabeled structures or size $n$ of structures in $\mbK_n$. The reason for assuming structures to grow becomes clear in the following lemma.
\begin{lma}\label{growingprop} $\mbK$ has a $0-1$ law and for some $m\in\mbbN$, $\lim_{n\rightarrow\infty} \mu_n(\{\mcM\in\mbK_n : |\mcM| \geq m\}) = 0$ if and only if there is a finite structure $\mcA$ such that almost surely for $\mcB\in\mbK$, $\mcA\cong \mcB$.
\end{lma}
\begin{proof} If almost surely all structures in $\mbK$ are isomorphic to $\mcA$ then everything true in $\mcA$ has probability one. Thus we have a $0-1$ law and choosing $m = |A|+1$ implies that $\lim_{n\rightarrow\infty} \mu_n(\{\mcM\in\mbK_n : |\mcM| \geq m\}) = 0$. For the other direction of the proof let $m\in\mbbN$ be the smallest number such that $\lim_{n\rightarrow\infty} \mu_n(\{\exists^{\geq m}x (x=x)\}) = 0$.
Since the vocabulary is finite there are only a finite amount of structures $\mcA_1,\ldots,\mcA_k$ of size smaller than $m$ up to isomorphism. However if we again apply the $0-1$ law, almost surely one of these structures must be isomorphic to the structures in $\mbK$.
\end{proof}

\noindent Equivalence relations definable in models of almost sure theories have a very special function, since an equivalence relation induces a partition on the universe of each structure almost surely. A partition of a finite structure become especially interesting if the equivalence classes have a fixed size since then we get information of how large the universe has to be in order to be partitioned in this way. We formalize these thoughts in the following lemma which is written in a very general context that will be useful later. 

\begin{lma}\label{eqlemma}
Assume $\mbK$ has a $0-1$ law, $\mcM\models T_\mbK$ and let $D=\xi(\mcM)$ (possibly empty) be an $\emptyset-$definable subset of $M^t$. 
Assume that for each $\bar a\in D$ the formula $\psi(x, y,\bar a)$ defines an equivalence relation $E$ on a set $A=\varphi(\mcM,\bar a)$ in $\mcM$ such that the equivalence classes $E_1,\ldots,E_m,\ldots$ are finite, only attain a finite amount of sizes and the amount of equivalence classes and their size are the same for each $\bar a\in D$. 
Then almost surely for $\mcN\in\mbK$ we have that $gcd(|E_1|,\ldots,|E_m|,\ldots)$ (the greatest common divisor) divides $|\varphi(\mcN,\bar b)|$ for each $\bar b$ such that $\mcN\models \xi(\bar b)$.
\end{lma}
\begin{proof}
Assume that the equivalence classes of $E$ have sizes $e_1,\ldots,e_n$ and $D\neq \emptyset$. Since $E$ is defined by $\psi(x,y,\bar a)$ for each $\bar a\in D$ there is a (parameter free) sentence $\psi_E$ which says that for each parameter in $D$, $E$ is an equivalence relation on $A$ and its equivalence classes only attain the fixed finite sizes. The sentence $\psi_E$ will look as follows, where $\psi_{eq}(\bar z)$ states that $\psi(x,y,\bar z)$ is an equivalence relation,
\[\forall \bar z  \Big(\xi(\bar z) \rightarrow \psi_{eq}(\bar z) \wedge \forall x \Big( \varphi(x,\bar z) \rightarrow \bigvee_{i=1}^n\exists^{=e_i}y(\varphi(y,\bar z)\wedge\psi (x,y,\bar z)) \Big)\Big).\]
Since $T_\mbK$ is complete and $\mcM\models \psi_E$ we see that $\psi_E\in T_\mbK$ and so $\psi_E$ is almost surely true in $\mbK$. Hence almost surely $\mcN\in\mbK$ defines, for each parameter $\bar b\in \xi(\mcN)$, an equivalence relation relation $\psi(x,y,\bar b)$ on the set defined by $\varphi(x,\bar b)$, with equivalence classes of size $e_1,\ldots,e_n$. But $\psi(x,y,\bar b)$ partitions $\varphi(\mcN,\bar b)$ so we see that if $c_i>0$ (we know this almost surely since it is true in $\mcM$) is the number of equivalence classes of size $e_i$ in $\varphi(\mcN,\bar b)$ then
\[|\varphi(\mcN,\bar b)| = \sum_{i=1}^n c_i e_i = gcd(e_1,\ldots,e_n)\frac{\sum_{i=1}^n c_ie_i }{gcd(e_1,\ldots,e_n)}.\]
 Hence $gcd(e_1,\ldots,e_n)$ divides $|\varphi(\mcN,\bar b)|$ almost surely and the lemma now follows. Note that if $D=\emptyset$ then the whole proof can be done in the same way except that we do not mention any parameter $\bar z$ in the first equation, and remove each instance of mentioning $\xi$, thus $\psi_E$ becomes simpler.
\end{proof}
\noindent By using the parameter free formula $x=x$ to define our target set $A$ and choosing an $\emptyset-$definable equivalence relation $E$, we get the following corollary.
\begin{cor}\label{eqcor1}
Assume $\mbK$ has a $0-1$ law and there is an $\emptyset-$definable equivalence relation $E$ on $\mcM\models T_\mbK$ such that the equivalence classes $E_1,\ldots,E_n,\ldots$ are finite, and only attain a finite amount of sizes. Then almost surely for $\mcN\in\mbK$ we have that $gcd(|E_1|,\ldots,|E_n|,\ldots)$ divides $|\mcN|$.
\end{cor} 
\noindent If we have an $\emptyset-$definable pregeometry, such as in simple $\omega-$categorical structures with SU-rank 1 (compare Fact \ref{simppregfact}) then there exists an equivalence relation on all objects outside $cl(\emptyset)$ by relating closed sets of a certain dimension.  
The next proposition, which is an application of Lemma \ref{eqlemma}, could be generalized even more to pregeometries definable using some parameter set which is $\emptyset-$definable (a formulation like in Lemma \ref{eqlemma}). This is however not necessary for our later applications and hence we write it in a bit more readable format.
\begin{lma}\label{cllemma}Assume $\mbK$ has a $0-1$ law, $\mcM\models T_\mbK$ and $n\in\mbbZ^+$. If a pregeometry is $\emptyset-$definable on an $\emptyset-$definable set $A = \varphi(\mcM)$ such that for each set $\{a_1,\ldots,a_n\},\{b_1,\ldots,b_n\}\subseteq A$ of independent elements $|cl(a_1,\ldots,a_n)| = |cl(b_1,\ldots,b_n)|$ and $|cl(a_1,\ldots,a_{n-1})| = |cl(b_1,\ldots,b_{n-1})|$. \\\indent  Then $|cl(a_1,\ldots,a_n)|-|cl(a_1,\ldots,a_{n-1})|$ almost surely divides \\$|\varphi(\mcN)|-|\cl(a_1,\ldots,a_{n-1})| $ for $\mcN\in\mbK$ and any independent set $\{a_1,\ldots,a_n\}\subseteq\varphi(\mcM)$. \end{lma}

\begin{proof}
Let $D$ be the $\emptyset-$definable set of $(n-1)$-tuples which consist of independent elements in $\varphi(\mcM)$. Then the formula
\[\psi(x,y,\bar z) \Longleftrightarrow cl(y,\bar z) = cl(x,\bar z) \]
defines an equivalence relation, for each tuple $\bar z\in D$, on the $\bar z-$definable set $A_0=\{a\in \varphi(\mcM) : a\notin cl(\bar z)\}$. Notice that each equivalence class of this relation has size $|cl(x,\bar z)| - |cl(\bar z)|$, a number which does not depend on $\bar z\in D$ or $x\in A_0$ by our assumptions in this lemma. Each equivalence class is finite. Lemma \ref{eqlemma} implies that $|cl(x,\bar z)| - |cl(\bar z)|$ almost surely divides $|A_0| = |\varphi(\mcN)|-|cl(\bar z)|$ for any independent $(n-1)-$tuple $\bar z$ of elements from $A$.
\end{proof}
\noindent In practice, many pregeometries found in structures are vector space pregeometries (or affine/projective). This gives us even more information to use and we may thus create the following corollary.
\begin{cor}\label{vectcor}
Assume $T_\mbK$ is $\omega-$categorical, $\mcM\models T_\mbK$ and $(M,acl)$ is a pregeometry isomorphic to the pregeometry of a vector space $V$ over a finite field $\mathbbm{F}$ equipped with the linear $span$ operator. Then there are $p,m\in\mbbZ^+$ such that $p$ is prime and for each $n\in\mbbZ^+$, $p^{mn}$ almost surely divides $|N|$ for $\mcN\in\mbK$. \end{cor}
\begin{proof}
Since $T_\mbK$ is $\omega-$categorical, the pregeometry which $(M, acl)$ defines is $\emptyset-$definable. By a well known characterization, a finite field has size $p^m$ for some prime $p$ and number $m$. In a vector space $V$ over $\mathbbm{F}$ we have that $|span(v_1,\ldots,v_r)| =|span(w_1,\ldots,w_r)|$ for any $r\in\mbbN$ and independent vectors $v_1,\ldots,v_r,w_1,\ldots,w_r\in V$, hence this is also true for the pregeometry $(M,acl)$. For each $n\in\mbbN$, Lemma \ref{cllemma} gives us that $|acl_\mcM(a_1,\ldots,a_{n+1})|-|acl_\mcM(a_1,\ldots,a_{n})|$ almost surely divides $|N| - |acl_\mcM(a_1,\ldots,a_{n})|$ for $\mcN\in\mbK$ and independent elements $a_1,\ldots,a_{n+1}\in M$. Thus
\[|acl_\mcM(a_1,\ldots,a_{n+1})|- |acl_\mcM(a_1,\ldots,a_{n})| = (p^m)^{n+1} - (p^m)^{n} = p^{mn}(p^m-1).\]
We conclude that some $k\in\mbbN$, $p^{mn}(p^m-1)\cdot k = |N|-|acl_\mcM(a_1,\ldots,a_n)| = |N|- p^{mn}$ which we can rewrite as $p^{mn}((p^m-1)\cdot k + 1) = |N|$, hence $p^{mn}$ divides $|N|$. 
\end{proof} 
\noindent If we are in the case of Corollary \ref{vectcor}, then we will in $\mbK$ not just have a growth, but structures will asymptotically grow in larger and larger steps. 
 This motivates why using a measure which depends on dimension and not size, as in \cite{AK,K}, is necessary if we want an almost sure theory with interesting pregeometries. We finish this section with yet an other corollary regarding almost sure theories.
\begin{cor}
Assume $T_\mbK$ is $\omega-$categorical, $\mcM\models T_\mbK$ and for each $n\in\mbbN$ and $\mcN\in\mbK_n$ we have that $|N|=n$. Then each $\emptyset-$definable equivalence relation in $\mcM$ with only finite sized equivalence classes has only a finite amount of different sizes and the sizes are relatively prime .
\end{cor}
\begin{proof}
From the $\omega-$categoricity and Ryll-Nardzevskis theorem it follows that an equivalence relation with only finite sized equivalence classes has to have a finite amount of different sizes.
Let $e_1,\ldots,e_n$ be the different sizes of equivalence classes. We may apply Corollary \ref{eqcor1} to see that $gcd(e_1,\ldots,e_n)$ divides $|N|$ almost surely for $\mcN\in\mbK$. But since $\mcN\in\mbK_m$ implies that $|N|=m$ for each $m\in\mbbN$, we see that $gcd(e_1,\ldots,e_n)$ has to divide more than two prime numbers, which is impossible unless $gcd(e_1,\ldots,e_n)=1$.\end{proof}

\section{$\omega-$categorical theories}\label{omegasection}
\noindent In this section we study the $\omega-$categorical theories from the view of extension properties, as defined bellow. These concepts are inspired by the method used to prove $0-1$ laws, originally used by Fagin \cite{F}, and we will extend the concepts in order to give further general results regarding $0-1$ laws and $\omega-$categorical almost sure theories.
\begin{defi}\label{thetadef} Let $T$ be a theory. For each $k\in\mbbZ^+$ assume that there are formulas $\theta_{k,1}(x_1,\ldots,x_k),\ldots,\theta_{k,i_k}(x_1,\ldots,x_k)$ such that if 
\[\sigma_k \equiv \forall x_1,\ldots,x_k \Big((\bigwedge_{1\leq i<j\leq k} x_i\neq x_j) \rightarrow \bigvee_{n=1}^{i_k} \theta_{k,n}(x_1,\ldots,x_k)\Big)\]
then $T\models \sigma_k$. Furthermore assume that for each $i\in\{1,\ldots,i_k\}$ there are \textbf{associated numbers} $j_1,\ldots,j_{m}\in\{1,\ldots,i_{k+1}\}$ for some $1\leq m\leq i_{k+1}$ such that if we, for each $j\in\{j_1,\ldots,j_m\}$, define formulas
\[ \tau_{k,i,j}  \equiv \forall x_1,\ldots,x_k (\theta_{k,i}(x_1,\ldots,x_k)\rightarrow \exists z \theta_{k+1,j}(x_1,\ldots,x_k,z)) \text{   and}\]
\[\xi_{k,i}\equiv\forall x_1,\ldots,x_k \Big(\theta_{k,i} (x_1,\ldots,x_k) \rightarrow \forall y \big(\bigwedge_{n=1}^k y\neq x_n  \rightarrow \]\[\bigvee_{n=1}^m \theta_{k+1,j_n} (x_1,\ldots,x_k, y)\big)\Big)\]
with the special case 
\[\tau_{0,1,j} \equiv \exists z\theta_{1,j}(z),\]
then $T\models \xi_{k,i}$ and $T\models \tau_{k,i,j}$. 
If all the above assumptions are true, for each $k$ and $i\in [i_k]$
\[T\models \forall x_1,\ldots,x_k (\theta_{k,i}(x_1,\ldots,x_k)\rightarrow\bigwedge_{1\leq\alpha < \beta\leq k} x_\alpha\neq x_\beta) \]
and for each $k$ and $i\in [i_k]$ there is an atomic diagram $R$ such that
\begin{equation}~\label{thetastrong} T\models\forall x_1,\ldots,x_k\Big (\theta_{k,i}(x_1,\ldots,x_k)\rightarrow R(x_{1},\ldots,x_{k})\Big), \end{equation}
then we say that $T$ satisfies \textbf{extension properties}. The formulas $\theta_{k,i}$ are called \textbf{extension axioms}.
\\\indent If $\mbK= (\mbK_n,\mu_n)_{n\in\mbbN}$ are sets of finite structures with an associated probability measure we say that $\mbK$ \textbf{almost surely satisfies extension properties} if its almost sure theory $T_\mbK$ satisfy extension properties.
\end{defi}
The notion of general extension properties is not new, they have been seen before in, for instance, Spencer's book \cite{S}. However when extension properties are formulated in the specific manner of Definition \ref{thetadef} they are closely connected to $\omega-$categorical theories which we will show in the following fact using the Ryll-Nardzevski theorem.
\begin{fact}\label{omegathm} Let $T$ be a theory. $T$ is $\omega-$categorical if and only if $T$ satisfies extension properties.
\end{fact}
\begin{proof}
 Assume $T$ is $\omega-$categorical. For each $n<\omega$, there exists a finite amount of complete $n-$types over $\emptyset$ of distinct tuples and they are isolated by formulas $\theta_{k,1}(\bar x),\ldots,\theta_{k,i_k}(\bar x)$. It is clear that property (\ref{thetastrong}) is satisfied and that for any $k-$tuple of distinct elements there exists a number $i$ (i.e. a complete type) such that the tuple satisfies $\theta_{k,i}$, thus $\sigma_k$ is satisfied. For each type $p(\bar x)$ there exists a distinct set of types $q_1,\ldots,q_r$ such that adding a distinct element to the tuple $\bar x$ implies that the tuple satisfies exactly one of $q_1,\ldots,q_r$. It is thus clear that there exist associated numbers $j_1,\ldots j_m$ such that $T\models\xi_{k,i}$ and for $j\in\{j_1,\ldots,j_m\} $, $T\models \tau_{k,i,j}$. \\\indent
For the other direction assume $T$ satisfies extension properties and let $\mcM,\mcN\models T$ be countable. Using the extension properties we will build an isomorphism $f$ between $\mcM$ and $\mcN$ in a back and forth way with functions $f_1,\ldots,f_n,\ldots$ such that $f_i$ is a partial isomorphism and if the domain of $f_k$ is $\bar x$, $\mcM\models\theta_{k,i}(\bar x)$ and $\mcN\models \theta_{k,i}(f_k(\bar x))$. For the base step, $f_1$ if we choose $a\in M$ then $\mcM\models \sigma_1$ implies that for some $i\in [i_1]$, $\mcM\models \theta_{1,i}(a)$, however as $\mcN\models \tau_{0,1,i}$ there is an element $b\in N$ such that $\mcN\models \theta_{1,i}(b)$. Define $f_1:\mcM\reduct \{a\}\rightarrow \mcN\reduct\{b\}$.\\\indent
Assume $f_k:\mcM\reduct\{a_1,\ldots,a_k\}\rightarrow\mcN\reduct\{b_1,\ldots,b_k\}$ respects $\theta_{k,i}$ as we described above and choose any $a_{k+1}\in M-\{a_1,\ldots,a_k\}$ (parallel reasoning if we chose $b_{k+1}\in N-\{b_1,\ldots,b_k\}$ first instead). As $T\models\xi_{k,i}$ there is a $j$ such that $\mcM\models\theta_{k+1,j}(a_1,\ldots,a_{k+1})$ and $\mcN\models\tau_{k,i,j}$. Thus there exists an element $b_{k+1}\in N$ such that $\mcN\models\theta_{k+1,j}(b_1,\ldots,b_{k+1})$. It is clear from property (\ref{thetastrong}) that if we extend $f_k$ to a map $f_{k+1}$ which takes $a_{k+1}$ to $b_{k+1}$ then this also is a partial isomorphism.
\end{proof}
The back and forth way of proving the previous fact implies that we may create an automorphism inside a structure satisfying extension properties such that any two tuples satisfying the same extension axioms are mapped to each other. Thus we conclude the following corollary.
\begin{cor}\label{exttpcor}
If $T$ satisfies extension properties, $\mcM\models T$ and $\bar a,\bar b\in M$ are such that for some extension axiom $\mcM\models \theta_{k,i}(\bar a)\wedge \theta_{k,i}(\bar b)$ then $tp(\bar a) = tp(\bar b)$.
 \end{cor}
On the other hand, regarding $0-1$ laws, we get the following corollary.
\begin{cor}\label{01cor}
Let $\mbK=(\mbK_n,\mu_n)_{n\in\mbbN}$. $T_\mbK$ is $\omega-$categorical and $\mbK$ has a $0-1$ law if and only if $\mbK$ almost surely satisfies extension properties
\end{cor}
\noindent This corollary gives a context to many previous results about $0-1$ laws, and shows that the method of using extension properties always works when proving a $0-1$ law if you have an $\omega-$categorical almost sure theory. There are though classes with $0-1$ laws without an $\omega-$categorical almost sure theory, such as if we let $\mbK_n$ consist of the graph with $n$ nodes in a cycle.
\begin{exa}\label{omegaex}
Let $\mbK_n$ be all relational structures of size $n$ over a vocabulary $V$ equipped with the uniform measure $\mu_n$, then $(\mbK_n,\mu_n)_{n\in\mbbN}$ has a $0-1$ law and the almost sure theory is $\omega-$categorical. Fagin \cite{F} proved this using extension axioms, in which he has $\theta_{k,i}(\bar x)$ as the quantifier free formula which describes the isomorphism type of a finite $V-$structure, letting $\theta_{k,1},\ldots,\theta_{k,i_k}$ enumerate all $V-$structures with universe $[k]$.
\end{exa}
Further examples, with more interesting properties, will be provided in Example \ref{partexa} and \ref{nonrigidexa}.

\begin{defi}
If $A\subseteq M^{eq}$ and there are only finitely many $\emptyset-$definable equivalence relations $E$ on $\mcM$ such that $P_E(\mcM^{eq})\cap A \neq \emptyset$ and for each such equivalence relation $P_E(\mcM^{eq})\subseteq A$  then we say that $A$ is a \textbf{full finitely sorted set}.
\end{defi}
Finite model theory does not have a counterpart to $\mcM^{eq}$ as $\mcM^{eq}$ always is an infinite structure. We will now however construct a way in which finite models may approximate parts of $\mcM^{eq}$  if its theory is an $\omega-$categorical almost sure theory.
\begin{defi}
Let $\mbK_n$ be a set of finite $V$-structures with a probability measure $\mu_n$, let $\mbK = (\mbK_n,\mu_n)_{n\in\mbbN}$ and let $E =\{E_1,\ldots,E_n\}$ be a set of $V-$formulas with even arities $2e_1,\ldots,2e_n$ respectively and define the vocabulary $V' = V \cup \{R_{E_i}, P_{E_i} : 1\leq i\leq n\}$ where $P_{E_i}$ is unary and $R_{E_i}$ has arity $e_i+1$. Notice that we may consider $V'\subseteq V^{eq}$. Then for each $n$ and $\mcN\in\mbK_n$ associate a structure $\mcN'$ in the following way: 
\begin{itemize}
\item If there is a formula in $E$, such that $\mcN$ does not interpret it as an equivalence relation, then expand $\mcN$ to $\mcN'$ as a $V'-$structure by interpreting each new relation symbol as $\emptyset$.
\item If each formula in $E$ is interpreted as an equivalence relation in $\mcN$ then let $A\subseteq \mcN^{eq}$ be the full finitely sorted set which contains the home sort and all equivalence classes of the formulas in $E$. Let $\mcN'=(\mcN^{eq}\upharpoonright A)\upharpoonright V'$. So $\mcN'\models R_{E_i}(\bar x,y)$ iff $\bar x$ is in the $E_i-$equivalence class $y$, and $\mcN'\models P_{E_i}(y)$ iff $y$ is an $E_i-$equivalence class.
\end{itemize} 
Let for each $n\in\mbbN$, $\mbK^E_n = \{\mcN' : \mcN\in \mbK_n\}$ and equip $\mbK^E_n$ with a probability measure $\mu^E_n$ by inducing it from the probability measure of $\mbK_n$ i.e. $\mu_n^E(\mcN') = \mu_n(\mcN)$.
\end{defi}
\noindent We define $\mbK^E_n$ so that it can use a finite slice of $\mcM^{eq}$, adding the equivalence classes to the structures in $\mbK_n$. In case $(\mbK_n,\mu_n)_{n\in\mbbN}$ has a zero-one law but at least one formula in $E$ is not almost surely an equivalence relation then $\mbK^E$ will almost surely be $\mbK$ (i.e. if $\mcN' \in\mbK^E$ then almost surely $\mcN\upharpoonright V \in\mbK$).
\begin{prp}\label{01eqthm}
Let $\mbK = (\mbK_n,\mu_n)_{n\in\mbbN}$ be a set of finite relational structures with almost sure theory $T_\mbK$. For a finite set of $\emptyset-$definable equivalence relations $E = \{E_1,\ldots,E_r\}$ on $\mcM\models T_\mbK$ let $\mbK^E = (\mbK_n^E,\mu_n^E)_{n\in\mbbN}$. Then the following are equivalent :
\begin{itemize}
\item[i)]$\text{$\mbK$ has a } 0-1 \text{ law and } $ $T_\mbK \text{ is } \omega-\text{categorical} $
\item[ii)] $\mbK^E$ has a $0-1$ law and $T_{\mbK^E}$ is $\omega-$categorical.
\end{itemize}
\end{prp}
\begin{proof}
The direction ii) implies i) is obvious so we focus on the case when $T_\mbK$ is $\omega-$categorical and $\mbK$ has a $0-1$ law.  
By Fact \ref{omegathm} we then know that there are $\theta_{k,i}$, $\xi_{k,i}$,$\sigma_k$ and $\tau_{k,i,j}$ so that $\mbK$ satisfies the extension properties using these. In order to prove this theorem we will modify the extension axioms on $\mbK$ so that we get new extension axioms which prove that $\mbK^E$ also satisfies extension properties. Then we use Fact \ref{omegathm} again to get the $0-1$ law and $\omega-$categoricity.\\\indent
Notice that from Corollary \ref{exttpcor} we may assume that for each formula $\theta_{k,i}$, $n_0\in[r]$, tuple $\bar x$ and elements in the tuple $x_{\alpha_1},\ldots,x_{\alpha_m}$,$x_{\alpha_{m+1}},\ldots,x_{\alpha_{2m}}$ we have either
\begin{equation}\label{omeq}
\mcM\models \theta_{k,i}(\bar x)\rightarrow E_{n_0}(x_{\alpha_1},\ldots,x_{\alpha_m},x_{\alpha_{m+1}},\ldots,x_{\alpha_{2m}}) \text{ or}
\end{equation}
\[\mcM\models \theta_{k,i}(\bar x)\rightarrow \neg E_{n_0}(x_{\alpha_1},\ldots,x_{\alpha_m},x_{\alpha_{m+1}},\ldots,x_{\alpha_{2m}}). \]
For each formula $\theta_{k,i}$ create a formula $\theta_{k,i}'$ by for each $\forall x \varphi(x)$ in $\theta_{k,i}$ change it to $\forall x (P_= (x) \rightarrow \varphi(x))$ and for each $\exists x\varphi(x)$ change it to $\exists x (P_=(x)\wedge\varphi(x))$. These new $\theta_{k,i}'$ make $\mbK^E$ satisfy extension properties on the home sort, but we still need something which considers the newly added imaginary sorts. 
Let $A\subseteq M^{eq}$ be the full finitely sorted set which contains all elements representing the equivalence relations in $E$ and the home sort. 
Each tuple $\bar c\in \mcM^{eq}\upharpoonright A$ may be written as $\bar x = \bar a\bar b_1\ldots\bar b_r$ up to permutation where all elements in $\bar a$ are from the home sort, and each element in $b\in\bar b_i$ satisfies $\mcM^{eq}\models P_{E_i}(b)$. From the $\omega-$categoricity we know that there are only a finite amount of tuples $\bar a$ in the home sort, up to type. For each type of a tuple $\bar a$ there are only a finite amount of ways the elements may stand in a relation to an imaginary element.
Hence the number of ways, which we have elements from the home sort satisfying $\theta_{l,j}$ and imaginary elements in relation to the elements in the home sort, is finite for $k-$tuples $\bar a\bar b_1\ldots\bar b_r\in\mcM^{eq}$.
This gives us the ability to create a finite amount of $\theta_{k,i}^e$ extension axioms by letting it stand for the formula which is a conjunction of $\theta_{l,j}', P_{E_\alpha}$  and $R_{E_\alpha}$ for the appropriate parts of a tuple.
Now what we have left to prove in this theorem is that these new formulas $\theta^e$ and corresponding $\tau^e_{k,i,j},\xi^e_{k,i}$ and $\sigma^e_k$ are almost surely true in $\mbK^E$. This may be shown through technical and tedious yet straight forward arguments, where the zero-one law of $\mbK$ and equation (\ref{omeq}) are the key elements in order to handle the imaginary elements. The details are left for the reader.
\end{proof}
\noindent It is clear from the axiomatization in the proof that we get the following corollary.
\begin{cor}\label{meqcor}
Assume $\mbK=(\mbK_n,\mu_n)_{n\in\mbbN}$, $T_\mbK$ is $\omega-$categorical with $\mcM\models T_\mbK$ and $A\subseteq \mcM^{eq}$ is a full finitely sorted set with equivalence relations $E =\{E_1,\ldots,E_r\}$ having classes represented in $A$. Then $Th((\mcM^{eq}\reduct A)\reduct V') = T_{\mbK^E}$.
\end{cor}
\begin{rmk}Using the previous proposition we can show the existence of sets of structures whose almost sure theory has an arbitrarily large finite SU-rank.
Assume $(\mbK_m,\mu_m)_{m\in\mbbN}$ has a $0-1$ law with an $\omega-$categorical simple almost sure theory with SU-rank $1$. Examples of such are, among others (see Section \ref{smsection} for more), the random triangle free graph or the random graph. For some $n\in\mbbN$ let $E_n$ be the equivalence relation $(x_1,..,x_n)E_n (y_1,\ldots,y_n)$ if and only if 
\[\big(\bigwedge_{i=1}^n\bigwedge_{j=1}^n (x_i=x_j \wedge y_i=y_j)\big) \vee \big(\bigwedge_{i=1}^n\bigwedge_{j=1}^n (x_i\neq x_j \wedge y_i\neq y_j \wedge x_i=y_i)\big)\]
Let $A\subseteq \mcM^{eq}$ be the full finitely sorted set which contain only the home sort and the equivalence classes for $E_n$. It is easy to verify that $\mcM^{eq}\reduct A$ have SU-rank $n$, since the elements representing the equivalence classes of $E_n$ will have SU-rank $n$. Corollary \ref{meqcor} thus implies that for each $m$, there exists a set of structures $\mbC_m$ with a probability measure $\tau_m$ such that $\mbC= (\mbC_m,\tau_m)_{m\in\mbbN}$ has a $0-1$ law and $T_\mbC$ is $\omega-$categorical and simple with SU-rank $n$. 
\end{rmk}
Proposition \ref{01eqthm} may of course also be used in order to prove $0-1$ laws or get nicer extension axioms for the almost sure theories by, after finding an almost sure equivalence relation $E$, converting from $\mbK$ to $\mbK^E$. 
\begin{exa} Let $\mbK_n$ consist of all labeled bipartite graphs with universe $[n]$ under the uniform measure $\mu_n$. Then $\mbK = (\mbK_n,\mu_n)_{n\in\mbbN}$ has a $0-1$ law and its almost sure theory is $\omega-$categorical by Kolaitis, Prömel, Rothschild \cite{KPR}, but the extension axioms are a bit complicated and speak about an almost surely $\emptyset-$definable equivalence relation $E$ which defines the two parts of a bipartite graph. If we instead extend $\mbK$ to $\mbK^E$, so each bipartite graph $\mcN\in\mbK$ gets two elements which points at the equivalence classes, then the extension axioms suddenly become very simple. We only need to check for elements $x,y\in N'$ if $\mcN'\models \forall z (R_E(x,z) \rightarrow R_E(y,z))$ holds or not. If it holds then no edges can exist between $x$ and $y$, and if it does not hold then we may have edges between them. \end{exa} We may also work in the opposite way. If we have a set of finite structures $\mbK=(\mbK_n,\mu_n)_{n\in\mbbN}$ and can identify some almost sure equivalence relations $E_1,\ldots,E_n$, then in order to find out if there is a $0-1$ law or not, we can transform $\mbK$ in to $\mbK^{\{E_1,\ldots,E_n\}}$ in order to possibly get an easier class to discuss and find out if there is convergence and $\omega-$categoricity or not.

\section{$\omega-$categorical simple theories with $SU-$rank 1}\label{smsection}
\noindent We assume, unless stated otherwise, that the vocabulary in this section is binary and relational.
The main goal of this section is to explore the $\omega-$categorical theories which in addition are simple or $\omega-$stable with SU$-$ rank $1$ and put these theories in the context of the extension properties of the previous section. The equivalence relation defined by $tp_{\mcM^{eq}}(x/acl_{\mcM^{eq}}(\emptyset)) = tp_{\mcM^{eq}}(y/acl_{\mcM^{eq}}(\emptyset)) $ is very important for these theories, however we will consider the abstract properties of it and use it in a more general form.
\begin{defi}\label{bgdef} Let T be a theory. We say that a formula $\xi$ is a \textbf{restricted equivalence relation} for $T$ if for some $k,t\in\mbbN$, $T$ implies that $\xi$ defines an equivalence relation with $k+t$ equivalence classes such that $k$ of the equivalence classes are infinite and $t$ of the equivalence classes have size $1$. The equivalence classes of $\xi$, if any, which have size $1$ are called \textbf{base sets}.
\end{defi}
For restricted equivalence relations we want to be able to fix what atomic diagrams are possible between and inside the classes. To do this we introduce the concept of a spanning formula, which is a formula stating the existence of all possible binary atomic diagrams with respect to equivalence classes.
\begin{defi}\label{spanningdef}
Let $T$ be a theory and $\xi$ a restricted equivalence relation for $T$ with $l$ equivalence classes. We say that a sentence $\gamma$ is \textbf{spanning $\xi$} if $T\models \gamma$ and the following holds: There are numbers $t_1,\ldots,t_l$ and a formula $\gamma_0$ such that $\gamma$ is equivalent with
\[\exists x_{1,1},\ldots,x_{1,t_1},\ldots,x_{l,1},\ldots,x_{l,t_l} \gamma_0.\]
The formula $\gamma_0$ in turn implies that if $m\neq p$ or $i\neq j$ then $x_{p,i}\neq x_{m,j}$. The $\xi-$equivalence class of $x_{m,j}$ is the same as $x_{p,i}$ if and only if $m=p$. Furthermore $\gamma_0$ implies that the atomic diagram of $x_{1,1},\ldots,x_{l,t_l}$ is fixed and for any elements $y,z$ there exists $m,p,i,j$ such that $y$ and $z$ are in the same $\xi-$equivalence class as $x_{m,i}$ and $x_{p,j}$, respectively, and $atDiag(x_{m,i},x_{p,j})= atDiag(y,z)$. 
\end{defi}
The following lemma is a direct consequence of the finiteness and the definitions.
\begin{lma}\label{spanninglma}
If $T$ is a complete theory over a finite vocabulary, then for each restricted equivalence relation $\xi$, there exists a formula $\gamma$ which is spanning $\xi$.
\end{lma}
We will now define the important concept of $\xi-$extension properties. This is essentially saying that between and inside equivalence classes we roll a die to determine binary atomic diagrams among pairs from a predetermined set. The trivial case with $\xi$ only having a single equivalence class would just be a random structure and if we look at only a single symmetric, anti-reflexive relation we get the random graph. 
\begin{defi}\label{xiextax}
Let $\xi(x,y)$ be a formula, $l\in\mbbN$ and $\Delta=\{\delta_{i,j}\}_{i,j\in [l]}$ where each $\delta_{i,j}$ is a non-empty set of binary atomic diagrams. For each $k\in\mbbN$ let $i_k\in\mbbN$. The formulas $\{\theta_{k,i}(y_1,\ldots,y_k) : k\in\mbbN, i\in [i_k]\}$ are called \textbf{$(\xi,\Delta)-$extension axioms} if the following requirements are satisfied. There is a formula $\gamma$ equivalent to
$\exists x_{1,1},\ldots,x_{1,t_1},\ldots,x_{l,1},\ldots,x_{l,t_l} \gamma_0(x_{1,1},\ldots,x_{l,t_l})$ such that $\gamma_0$ implies that for each $i,j\in [l]$,  
$\{atDiag(x_{i,\alpha}, x_{j,\beta}) : \alpha \in [t_i], \beta\in [t_j]\} = \delta_{i,j}$
and $\xi(x_{i,\alpha},x_{j,\beta})$ holds if and only if $i=j$.
For each $k\in \mbbN, j_1,\ldots,j_k\in [l]$ and collection $\{\eta_{\alpha,\beta}\}_{\alpha,\beta\in \{j_1,\ldots,j_k\}}$ such that $\eta_{\alpha,\beta}\in\delta_{\alpha,\beta}$ there is $j\in [i_k]$ and a formula $\theta_{k,j}'$ such that $\theta_{k,j}'(y_1,\ldots,y_k,x_{1,1},\ldots x_{l,t_l})$ implies that for each $r,s\in [k]$, $atDiag(y_{r},y_{s}) = \eta_{i_{r},i_{s}}$ and $\xi(y_{r},x_{r,1})$ hold.
For each $k$ and $i\in[i_k]$, $\theta_{k,i}(y_1,\ldots,y_k)$ is equivalent to the formula 
\[\exists x_{1,1},\ldots,x_{1,t_1},\ldots,x_{l,1},\ldots,x_{l,t_l} \big(\gamma_0(x_{1,1},\ldots, x_{l,t_l}) \wedge \]\[\theta_{k,i}'(y_1,\ldots,y_k,x_{1,1},\ldots,x_{l,t_l})\big).\] 
We say that a theory $T$ satisfies \textbf{$(\xi,\Delta)-$extension properties} if $T$ implies that $\xi$ is a bounded equivalence relation with $l$ equivalence classes and $T$ satisfies extension properties using the $(\xi,\Delta)-$extension axioms as extension axioms according to Definition \ref{thetadef} with $\theta_{k,i}$ associated to $\theta_{k+1,j}$ if $\theta_{k,i}'$ is a subformula of $\theta_{k+1,j}'$. We may use the term \textbf{$\xi-$extension properties} to indicate $(\xi,\Delta)-$extension properties for some set $\Delta$ containing sets of binary atomic diagrams.
\end{defi}
Although the definition may seem overly technical, these kind of extension properties have been used before. We give a few examples to showcase this and to display how the three previous definitions work in practice.
\begin{exa}\label{partexa}
In \cite{C} Compton looked at $\mbK_n$ as consisting of all (labeled) partial orders of size $n$ and showed that $\mbK =(\mbK_n,\mu_n)_{n\in\mbbN}$ has a $0-1$ law if $\mu_n$ is the uniform measure. This proof was done by first using a result by Kleitman and Rothschild \cite{KR}, who proved that almost surely all partial orders have height exactly $3$ i.e. we may divide the partial orders in to a top, a bottom and a middle layer of elements. This property may be described by an $\emptyset-$definable equivalence relation $\xi$ for $T_\mbK$, which thus is restricted with no base sets. Compton then used the following properties:
\begin{itemize}
\item For any finite disjoint sets $X,Y$ of middle elements, there there are elements $a,b$ such that $a$ is greater than each element in $X$, but unrelated to $Y$ and $b$ is less than each element in $X$, but unrelated to $Y$.
\item For each disjoint set $X_0,Y_0$ of top elements and $X_1,Y_1$ of bottom elements there is an element $c$ such that $c$ is in the middle layer between $X_0$ and $X_1$, but unrelated to $Y_0$ and $Y_1$.
\end{itemize}
Put into the terms of this article, Compton showed that $\mbK$ almost surely satisfy $\xi-$extension properties. It thus becomes clear from Theorem \ref{simpthm} that $T_\mbK$ is $\omega-$categorical and simple with SU-rank $1$. This is a sharp contrast to the homogeneous partial order, generated by taking the Fra\"{i}ss\'e limit of $\mbK$, which is clearly not simple since it satisfies the strict order property. The same phenomena has been noted in the sets of structures studied by Koponen \cite{K} and Mubayi and Terry \cite{MT}, however the general question when and why the Fra\"iss\'e limit and the probabilistic limit are the same remains open.
\end{exa}
\begin{exa}\label{nonrigidexa}
In \cite{AK2} the author together with Koponen showed that the set of all finite non-rigid structures $\mbK = (\mbK_n,\mu_n)_{n\in\mbbN}$ (structures with non-trivial automorphism group) equipped with the uniform measure $\mu_n$, do not have a $0-1$ law but a convergence law. Let $\mbS(\mcA, H)\subseteq \mbK$ be all structures in which the nonrigid finite structure $\mcA$ is embeddable into and which have an automorphism group containing $H$ as a subgroup such that all the elements in $\mcA$ are moved by some automorphism. $\mbS(\mcA, H)$ is shown to have a $0-1$ law by proving that $\mcA$ is almost surely definable and then creating $\xi-$extension axioms. The formula $\xi$ in this case will describe wether what relation it has to $\mcA$, hence distinguishing elements in $\mcA$. Thus a structure $\mcM$ satisfying the almost sure theory of $\mbS(\mcA, H)$ is $\omega-$categorical, simple with $SU-$rank 1, with trivial pregeometry and $acl(\emptyset)= \mcA$. Moreover if $X$ is the union of all infinite equivalence classes of $\xi$ then $\mcM\reduct X$ forms the structure which satisfies the almost sure theory of $\mbC_n$ consisting of all structures of size $n$ under the uniform measure. 
\\\indent The convergence law of $\mbK$ is then determined by looking at appropriate different $\mcA$ and $H$ and take the union of these $\mbS(\mcA, H)$. Thus the convergence law is determined by, in the almost sure theory of $\mbS(\mcA, H)$, what structure there is in $acl(\emptyset)$ and what atomic diagrams there are between $acl(\emptyset)$ and the rest of the structure.
\end{exa}

\begin{thm}\label{simpthm}
If V is binary and T is a complete theory then the following are equivalent.
\begin{itemize}
\item[(i)] $T$ is $\omega-$categorical, supersimple with SU-rank $1$ and has trivial pregeometry
\item[(ii)] There is a restricted equivalence relation $\xi$ for $T$
such that $T$ satisfies $\xi-$extension properties.
\end{itemize}
\end{thm}
We will prove this theorem through the direct application of Lemma \ref{simpthmlma} and Lemma \ref{simplma}.
\begin{lma}\label{simpthmlma}
Assume $V$ is binary, $T$ is $\omega-$categorical, supersimple with SU-rank $1$ and with trivial pregeometry. Let $\mcM\models T$. If $\xi(x,y)$ is the equivalence relation defined by $tp(x/ acl_{\mcM^{eq}}(\emptyset))= tp(y/ acl_{\mcM^{eq}}(\emptyset))$ then $\xi$ is restricted and $T$ satisfies $\xi-$extension properties.
\end{lma}
\begin{proof}
Note that, since $SU(\mcM) = 1$, $\xi$ has only a finite amount of equivalence classes where all elements are inside an equivalence class which is either infinite or of size one. Thus $\xi$ is a restricted equivalence relation in $T$. For the rest of this proof assume $\xi$ has $l$ equivalence classes and enumerate them from $1$ to $l$. For each $i,j\in [l]$ let $\delta_{i,j}$ be the set of all binary atomic diagrams existing between elements in class $i$ and class $j$ (or between elements inside class $i$ if $i=j$) and put $\Delta=\{\delta_{i,j}\}_{i,j\in[l]}$. Using $\Delta$ and $\xi$ we may now create $(\Delta,\xi)-$extension axioms and it thus remains to prove that $Th(\mcM)$ satisfies $(\Delta,\xi)-$extension properties. We will use the terminology from Definition \ref{thetadef} in order to do the proof.
\\\indent It is clear from the definition of $\theta_{k,i}$ that for each $k\in\mbbN$ and $i\in [i_k]$,  $\mcM\models\sigma_{k} \wedge \xi_{k,i}$. Assume that $\mcM\models\theta_{k,i}(a_1,\ldots,a_k)$ and $j$ is an associated number to $i$. If $a \in acl(\emptyset)$ and $d_1,d_2\notin acl(\emptyset)$ but $\mcM\models \xi(d_1,d_2)$ then $d_1$ and $d_2$ have the same atomic diagram to $a$, and this fact is expressed by $\gamma$. We may thus assume without loss of generality that $\theta_{k+1,j}(y_1,\ldots, y_k,y_{k+1})$ implies that none of $y_1,\ldots,y_{k+1}$ is in the base set of $\xi$ i.e. in $acl(\emptyset)$. Further assume that $p$ is such that $\theta_{k+1,j}(y_1,\ldots,y_{k+1})$ implies that $\xi(x_{p,1},y_{k+1})$ hold.
That $\mcM\models\theta_{k,i}(a_1,\ldots,a_k)$ holds implies that for some element $d_{p,1}$, witnessing $x_{p,1}$ and each $s\in [k]$ there are elements $d_{m,\alpha_s}, d_{p,\beta_s}$ (witnessed by $\gamma$) such that $\mcM\models \xi(d_{p,\beta_s}, d_{p,1})\wedge \xi(a_s,d_{m,\alpha_s})$ and $atDiag(d_{m,\alpha_s},d_{p,\beta_s})= atDiag(y_s,y_{k+1})$, as implied by $\theta_{k+1,j}$. For each $s\in[k]$, $tp(a_s/ acl_{\mcM^{eq}}(\emptyset))= tp(d_{m,\alpha_s}/ acl_{\mcM^{eq}}(\emptyset))$. Thus there exists $c_1,\ldots,c_k$ such that for each $s\in[k]$, $atDiag(a_s,c_s) = atDiag(d_{m,\alpha_s},d_{p,\beta_s}) = atDiag(y_s,y_{k+1})$ and $\mcM\models \xi(c_s,d_{m,1})$. We may conclude that for each $s,r\in[k]$, $tp(c_s/ acl_{\mcM^{eq}}(\emptyset))= tp(c_r/ acl_{\mcM^{eq}}(\emptyset))$ and the distinct elements $c_1,\ldots,c_k,a_1,\ldots,a_k$ are all independent since $acl$ is trivial. The independence theorem (Fact \ref{indthm}) then implies that there exists an element $c$ such that $\mcM\models \xi(c,d_{m,1})$ and for each $s\in [k]$,$atDiag(a_s,c_s) = atDiag(a_s,c) = atDiag(y_s,y_{k+1})$. It follows that $\mcM\models \theta_{k+1,j}(a_1,\ldots,a_k,c)$ and hence we have shown that $\mcM\models \tau_{k,i,j}$, thus $T$ satisfies the $\xi-$extension properties.
\end{proof}From the previous proof we may deduce the following corollary which will be useful later.
\begin{cor}\label{simpcor}
Assume $V$ is binary, T is simple, $\omega-$categorical, SU(T)=1 and $acl$ is trivial.  Let $\mcM\models T$, let $\xi$ be the equivalence relation defined by $tp(x/acl_{\mcM^{eq}}(\emptyset)) = tp(y/acl_{\mcM^{eq}}(\emptyset))$, and $\Delta$ be the set of all sets $\delta_{i,j}$ of atomic diagrams between equivalence class $i$ and $j$. Then the set of $(\xi,\Delta)-$extension axioms axiomatizes $T$.
\end{cor}
To prove the second direction of Theorem \ref{simpthm} we create a small lemma. It is clear from the proof of this lemma that $acl(\emptyset)$ of any structure satisfying $\xi-$extension properties coincide with the base sets of $\xi$. Note that we do not use that we are working over a binary vocabulary explicitly in the proof and thus if we had defined extension properties for general vocabularies then this Lemma would still hold.
\begin{lma}\label{simplma} If there exists a restricted equivalence relation $\xi$ for $T$ such that $T$ satisfies $\xi-$extension axioms then $T$ is $\omega-$categorical, supersimple with $SU-$rank 1 and has trivial pregeometry.
\end{lma}
\begin{proof}
It is clear from Fact \ref{omegathm} that $T$ is $\omega-$categorical. That $T$ is supersimple with $SU-$rank 1 follows from a standard argument which we will sketch here. We claim that if $\mcM\models T$ and $\bar a\in M, A\subseteq M$ with $A_0 = \bar a\cap acl(A) $ then $\bar a\ind_{A_0} A$ which in turn implies what we want to prove.\\\indent
Assume $\bar a\nind_{A_0} A$ and hence $tp(\bar a/A)\models\varphi(\bar x,\bar b)$ such that $\varphi(\bar x,\bar b)$ divides over $A_0$. Assume that $\bar b_1,\bar b_2, \ldots$ is an indiscernible sequence such that $tp(\bar b/ A_0)=tp(\bar b_1/A_0)=\ldots$ and $\{\varphi(\bar x,\bar b_i): i=1,\ldots\}$ is $r-$inconsistent for some $r\in \mbbN$. Let $\bar c_1,\bar c_2, \ldots$ be tuples such that $\mcM\models \varphi(\bar c_j,\bar b_j)$. Since this is an infinite sequence there has to exist $\bar c_{i_1},\ldots,\bar c_{i_r}$ with the same atomic diagram such that each component in one tuple is in the same $\xi-$equivalence class as the corresponding component in the other tuples. But then the $\xi-$extension axioms implies that there exists a tuple $\bar c$ such that $\bar b_{i_j}\bar c_{i_j}$ has the same atomic diagram and $\xi-$classes as $\bar b_{i_j}\bar c$. Using Corollary \ref{exttpcor} it follows that $tp(\bar b_{i_j}\bar c_{i_j}) = tp(\bar b_{i_j}\bar c)$ for each $j\in[r]$. Hence for each $j\in [r]$ $\mcM\models \varphi(\bar c,\bar b_{i_j})$, which means that we have a contradiction against the $r-$inconsistence.\\\indent
Lastly we show that $acl$ is trivial. If $a\in M$ is part of the base set of $\xi$, then clearly $a\in acl(\emptyset)$. Assume distinct $b,\bar a\in M$ are both disjoint from the base sets and $b\in acl(\bar a)$. The $\xi-$extension properties however imply that there exist an arbitrary amount of elements $b_1,b_2,\ldots , b_n$ such that $b_i\bar a$ have the same atomic diagram as $b\bar a$ and are in the same respective equivalence class. But then $tp(b\bar a) = tp(b_i\bar a)$ and hence $b\notin acl(\bar a)$.
\end{proof}

A special case of being simple is to be $\omega-$stable from which we may deduce the following corollary to Theorem \ref{simpthm}.
\begin{cor}\label{stablethm}
Assume that $T$ is a complete theory over a binary vocabulary. The following are equivalent:
\begin{itemize}
\item[(i)] $T$ is $\omega-$stable, $\omega-$categorical with SU-rank $1$ and trivial pregeometry.
\item[(ii)] there is a restricted equivalence relation $\xi$ for $T$ and a set of sets of binary atomic diagrams $\Delta=\{\delta_{i,j}\}$ with $|\delta_{i,j}|=1$ for each $i,j$ such that $T$ satisfy $(\xi,\Delta)-$extension properties.
\item[(iii)] there is a restricted equivalence relation $\xi$ for $T$ such that if $\mcM\models T$ then each equivalence classes $X$ of $\xi$ is indiscernible sets over $M-X$.
\end{itemize}
\end{cor}
\begin{proof}
Assume (i) and let $\mcM\models T$. Lemma \ref{simpthmlma} implies that if $\xi (x,y)$ is the equivalence relation $tp(x/acl_{\mcM^{eq}}(\emptyset)) = tp(y/acl_{\mcM^{eq}}(\emptyset))$ then $T$ satisfies $(\xi,\Delta)-$extension axioms for some set $\Delta$. If $A\subseteq M$ and $a,b\in M-A$ such that $\mcM\models\xi(a,b)$ then $p(x)=tp(a/acl_{\mcM^{eq}}(\emptyset))=tp(b/acl_{\mcM^{eq}}(\emptyset))$ and by stability and $SU-$rank 1, there is thus a unique way to extend $p(x)$ to a type over $A$ hence $tp(a/A)=tp(b/A)$. This implies that the atomic diagram of $\{a\}\cup A$ is the same as for $\{b\}\cup A$ for any $A\subseteq M$. We may thus conclude that $|\delta_{i,j}|= 1$ for each $\delta_{i,j}\in \Delta$.

Assume (ii) in order to prove (iii). Let $\mcM\models T$, assume $a_1,\ldots,a_k\in M$ are in the same $\xi-$equivalence class and assume $b_1,\ldots,b_r\in M$ are not in the same class as $a_1$. If $c_1,\ldots,c_k$ are in the same class as $a_1$ then, by the assumptions, they satisfy the same extension axioms, i.e. $\mcM\models \theta_{k,i}(a_1,\ldots,a_k) \wedge \theta_{k,i}(c_1,\ldots,c_k)$ for some $i$. However there is a unique way to extend $c_1,\ldots,c_k$ to any element in the same equivalence class as $b_1$. Thus by induction there is $j$ such that $\mcM\models\theta_{k,j}(a_1,\ldots,a_k,b_1,\ldots,b_r) \wedge \theta_{k,j}(c_1,\ldots,c_k,b_1,\ldots,b_r)$ and hence $tp(a_1,\ldots,a_k/b_1,\ldots,b_r) = tp(c_1,\ldots,c_k/b_1,\ldots,b_r)$. 

If we assume (iii) and want to prove (i), assume $\mcM\models T$. For any $A\subseteq M$ it is clear from the assumption that the algebraic closure is trivial and for any tuple $\bar a\in M$ such that $\bar a\cap A =\emptyset$ the type $tp(\bar a / A)$ only depend on which $\xi-$equivalence class the elements of $\bar a$ are in. Thus we conclude that the SU-rank is $1$ and if $|A|=\aleph_0$ there are only $\aleph_0$ complete types over $A$, hence we have $\omega-$stability. $T$ is $\omega-$categorical since the type of a tuple only depend on which equivalence classes it belongs, and thus there are only a finite amount of $n-$types over $\emptyset$ for each $n<\omega$.
\end{proof}
As a special case of the $\omega-$stable theories we have the strongly minimal ones. 
\begin{cor}\label{smcor}
Assume that $T$ is a complete theory over a binary vocabulary. The following are equivalent.
\begin{itemize}
\item[(i)] $T$ is strongly minimal and $\omega-$categorical with trivial algebraic closure.
\item[(ii)] There is a restricted equivalence relation $\xi$ for $T$ with only one infinite equivalence class in $\mcM\models T$ such that $T$ satisfies $\xi-$extension properties and all pairs of elements which are not from the base sets have the same atomic diagram.
\item[(iii)] If $\mcM\models T$, there exists a cofinite $\emptyset-$definable set which is indiscernible over the rest of $\mcM$.
\end{itemize}
\end{cor}
\begin{proof} Assume that $\mcM\models T$ is strongly minimal and $\omega-$categorical, thus $\mcM$ is $\omega-$stable. Corollary \ref{stablethm} then implies that there is a restricted equivalence relation $\xi$ for which we satisfy $\xi-$extension properties. The strong minimality however implies that there is only one infinite equivalence class hence (ii) follows.\\\indent 
Assume (ii). The infinite equivalence class of $\xi$ is an $\emptyset-$definable set. By Corollary \ref{stablethm} this set is indiscernible over the rest of $\mcM$. If we assume (iii) it is clear that only a finite amount of $n-$types may exist over $\emptyset$ for each $n<\omega$, thus $T$ is $\omega-$categorical. By indiscernability either $\varphi(x,\bar a)$ is satisfied by all elements in the cofinite set (and not in $\bar a$) or none, thus $\varphi(x,\bar a)$ is defining a finite or cofinite set. Hence $T$ is strongly minimal.
\end{proof}

\begin{rmk}\label{highervocrmk}
It is quite clear that the definition of spanning formulas \ref{spanningdef} and $\xi-$extension properties \ref{xiextax} may be extended into the context of any finite relational vocabulary $V$. With these more general assumptions Corollaries \ref{stablethm} and \ref{smcor} have proofs which are very similar, though more technical, with the main component being the fact that $\omega-$stable theories have stationary types over algebraically closed sets. Theorem \ref{simpthm} however is not possible to generalize using our method as the independence property of simple theories is not strong enough to handle the higher arity relational symbols in a good enough way. 
\end{rmk}
The pregeometry defined by the algebraic closure in a strongly minimal $\omega-$categorical theory satisfies that if $X$ and $Y$ both are independent sets of equal size then $|cl(X)|=|cl(Y)|$. It thus follows, using Lemma \ref{cllemma}, that if $(\mbK_n,\mu_n)_{n\in\mbbN}$ is a class of structures such that almost surely $|N|= n$ for $\mcN\in\mbK_n$ and the almost sure theory $T_\mbK$ is strongly minimal and $\omega-$categorical then the algebraic closure is trivial. This conclusion combined with the previous remark gives us the following result. 
\begin{prp}\label{smlma}
Let $V$ be any finite relational vocabulary. Assume a set of $V-$structures $\mbK = (\mbK_n,\mu_n)_{n\in\mbbN}$ are such that $|N| = n$ almost surely for $\mcN\in \mbK_n$ then the following are equivalent:
\begin{itemize}
\item $T_{\mbK} \text{ is strongly minimal and $\omega-$categorical}$.
\item $\mbK$ has a $0-1$ law and there exists a number $m\in \mbbN$ such that almost surely for $\mcN\in\mbK$ there is $X\subseteq N$ with $|X|=m$ such that $N-X$ is indiscernible over $X$.
\end{itemize}
\end{prp}
The assumption that $\mbK$ must have a $0-1$ law is necessary for the second direction, as we may almost surely have indiscernible sets even though the same things are not almost surely true. An easy example is letting $\mbK_n$ consist of a complete graph on $n$ nodes when $n$ is even and the complement of the complete graph on $n$ nodes when $n$ is odd.

\section{Approximating theories using probabilities on finite structures}\label{classsection}
In this section we prove Theorems \ref{rngstrthm} and \ref{reductprp} which say that the simple $\omega-$categorical structures with SU-rank 1 and trivial pregeometry are possible to approximate using finite structures and almost sure theories. In order to do this, we will need to define very strict ways to uniformly create the finite structures and in this way satisfy the correct extension properties. We thus define a set of atomic diagrams such that it could have been gotten from an infinite structure with a restricted equivalence relation $Q$ as we saw in Section \ref{smsection}.
\begin{defi}
Assume that $l,t\in\mbbN$, $t<l$, $Q$ is a relational symbol in the vocabulary and $\Delta = \{\delta_{i,j}\}_{i,j\in[l]}$ is such that each $\delta_{i,j}$ is a set of binary atomic diagrams. We call $\Delta$ an \textbf{$(l,t, Q)-$compatible set} if the following properties are satisfied:
\begin{itemize}
\item For any $i,j\in [l]$ if $\zeta(x,y) \in\delta_{i,j}$ then $\zeta(y,x)\in \delta_{j,i}$.
\item For any $i,j,k\in [l]$ if $\zeta(x,y)\in\delta_{i,j}$ and $\zeta'(x',y')\in \delta_{i,k}$ then $\zeta(x,y)$ specifies $x$ to have the same unary atomic diagram as $x'$ in $\zeta'(x',y')$.
\item For any $i\in[l]$ if $j\in \{l-t+1,\ldots,l\}$ then $|\delta_{i,j}| = 1$.
\item $Q(x,y)\wedge Q(y,x)\wedge Q(x,x)$ hold in all atomic diagrams in $\delta_{i,i}$ and $\neg Q(x,y)\wedge \neg  Q(y,x) \wedge Q(x,x) $ hold in all atomic diagrams in $\delta_{i,j}$ if $i\neq j$.
\end{itemize}
\end{defi}
Using the $(l,t,Q)-$compatible sets we will now show a $0-1$ law which will be the foundation for the rest of this section.
The next proposition may seem easy to generalize to structures with more complex vocabulary than binary, however problems may arise with dependence between a lower arity relational symbol and a higher one, which seem to make things quite complicated. This may though be possible to fix by giving an even more elaborate definition than the one above.
If we assume that there is a unique atomic diagram between fixed classes, then a generalization of the proposition to higher arities becomes a quite trivial exercise.
\begin{prp}\label{lawprop}
Let $l,t\in\mbbZ^+$, $Q\in V$ and assume that $\Delta=\{\delta_{i,j}\}_{i,j\in[l]}$ is an $(l,t,Q)-$compatible set. If 
\[\mbK_n = \{\mcN : N = ([n]\times[l-t])\ \cup \ (\{1\}\times \{l-t+1,\ldots ,l\}) \text{ and if } (a,i),(b,j)\in N \]\[\text{ then } atDiag^\mcN((a,i),(b,j))\in \delta_{i,j}  \}\]
with associated uniform measure $\mu_n(\mcN) = 1/|\mbK_n|$ then $\mbK = (\mbK_n,\mu_n)_{n\in\mbbN}$ almost surely satisfies $(\Delta,Q)-$extension properties and has a $0-1$ law with an almost sure theory which is supersimple and $\omega-$categorical with SU-rank 1 and trivial pregeometry.
\end{prp}
\begin{proof} Note that $Q(x,y)$ form a restricted equivalence relation in $\mbK$, thus it follows from Lemma \ref{simplma} that if we can prove that $\mbK$ almost surely satisfies $(\Delta,Q)-$extension properties then $\mbK$ has a $0-1$ law and $T_\mbK$ satisfy all required properties. \\\indent 
If $i,j\in [l-t]$ and $\zeta\in \delta_{i,j}$ then 
\[\mu_n(\exists x,y (atDiag(x,y)=\zeta)) \approx 1 - \frac{(|\delta_{i,j}|-1)^{n^2}}{|\delta_{i,j}|^{n^2}}\] 
which tends to $1$ as $n\rightarrow \infty$. On the other hand if $i\in \{l-t+1,\ldots,l\}$ then $|\delta_{i,j}| = 1$. Thus we conclude that each atomic diagram in any $\delta_{i,j}$ has an asymptotic probability of $1$ to exist, and hence there exists a formula $\gamma$ which is spanning $Q$. 
\\\indent
Using $\Delta$ and $Q$ we may create $(\Delta,Q)-$extension axioms according to Definition \ref{xiextax} and hence we now need to prove that the properties in Definition \ref{thetadef} all almost surely hold in order to finish this proof. It is clear that the extension axioms satisfy property (\ref{thetastrong}). It remains to prove that the formulas $\sigma_k, \tau_{k,i,j}$ and $\xi_{k,i}$ hold.
It is clear, by the way the structures in $\mbK_n$ are defined, that almost surely for $\mcN\in \mbK_n$ we have $\mcN\models\sigma_k \wedge \xi_{k,i}$ for each $k\in\mbbN$ and $i\in [i_k]$.
\\\indent
If $\mcN\models \theta_{k,i}(a_1,\ldots,a_k)$ and $j$ is one of its associated numbers, $\theta_{k+1,j}$ have an equivalence class, with number $p$, (with respect to $\gamma$) pointed out for the extra element. If $p\in\{l-t+1,\ldots,l\}$ then, with probability one, $\mcN\models \tau_{k,i,j}$ since there is only one way for elements in specific equivalence classes to be adjacent to elements in the equivalence class $p$. Assume $p\in[l-t]$. The probability that no element with this atomic diagram exists for some elements satisfying $\theta_{k,i}$ is at most
\[\binom{n\cdot l}{k}\left(\frac{c-1}{c}\right)^{\frac{n}{l} - k -s}\]
where $c$ is the number of possible isomorphism classes of k+1 elements in the chosen equivalence classes and $s$ is the number of elements which $\gamma$ talk about. We note that this probability goes to $0$ as $n$ grows, and thus $\tau_{k,i,j}$ is almost surely true.
\end{proof}
We will now move on to proving Theorem \ref{rngstrthm}. It might seem like we, in the proof, are taking a huge detour to a new structure $\mcM'$. However the problem in studying $\mcM$ is that we do not know, without our detour to $\mcM'$, if the equivalence relation $\xi$ in $\mcM$ does almost surely define an equivalence relation with the desired properties in $\mbK$.
\begin{proof}[Proof of Theorem \ref{rngstrthm}]
Let $\mcM\models T$ and $\xi(x,y)$ be a formula representing the equivalence relation $tp_{\mcM^{eq}}(x/acl_{\mcM^{eq}}(\emptyset)) = tp_{\mcM^{eq}}(y/acl_{\mcM^{eq}}(\emptyset))$ which we know is $\emptyset-$definable and have $l$ equivalence classes out of which $t$ are finite. Let $V'= V\cup \{Q\}$ where $Q$ is a new binary relation and create the $V'$-structure $\mcM'$ such that $\mcM'\reduct V=\mcM$ and $\mcM\models \xi(a,b)$ if and only if $\mcM'\models Q(a,b)$. Obviously $Q$ is an equivalence relation in $\mcM'$ and since $Q$ only represents an $\emptyset-$definable relation in $\mcM$ it is clear that $tp_{\mcM'^{eq}}(x/acl_{\mcM'^{eq}}(\emptyset)) = tp_{\mcM'^{eq}}(y/acl_{\mcM'^{eq}}(\emptyset))$ if and only if $\mcM'\models Q(x,y)$. \\\indent
Let $\Delta'=\{\delta'_{i,j}\}_{i,j\in[l]}$ be such that for each $i,j\in [l]$, $\delta'_{i,j}$ is the set of all binary atomic diagrams existing between elements in $Q-$equivalence class $i$ and $j$ in $\mcM'$. Note that if $a,b,c \in M'$ with $a\in acl_{\mcM'}(\emptyset)$ and $\mcM'\models Q(b,c)$ then $tp_{\mcM'}(b/a)= tp_{\mcM'}(c/a)$ and thus if the equivalence class of $a$ and $b$ is $i$ and $j$ respectively then $|\delta'_{i,j}|= 1$. The remaining properties of $\Delta$ are clear and we may thus conclude that $\Delta$ is an $(l,t,Q)-$compatible set.
Define
\[\mbK'_n = \{\mcN : N = ([n]\times[l-t])\ \cup\ (\{1\}\times \{l-t+1,\ldots, l\}) \text{ and if } (a,i),(b,j)\in N \]\[\text{ then } atDiag((a,i),(b,j))\in \delta_{i,j}  \}\]
and associate the uniform probability measure $\mu'_n(\mcM)= 1/|\mbK'_n|$ with it.
From Proposition \ref{lawprop} we get that $\mbK'=(\mbK_n', \mu_n')_{n\in\mbbN}$ has a $0-1$ law and $T_{\mbK'}$ satisfy $(\Delta', Q)-$extension properties. By Corollary \ref{simpcor}, $Th(\mcM')$ is axiomatized by $(\Delta', Q)-$extension properties and thus $Th(\mcM') = T_{\mbK'}$. By definition $\mcM'\models \forall x,y (\xi(x,y)\leftrightarrow Q(x,y))$ thus $\xi$ is almost surely a restricted equivalence relation in $\mbK'$ and $\mbK'$ almost surely satisfy $(\Delta',\xi)-$extension properties. \\\indent
Let $\Delta = \Delta'\reduct V$, $\mbK_n = \{\mcN'\reduct V  : \mcN'\in \mbK_n'\}$ with probability measure $\mu_n (\mcN\reduct V)= \mu_n'(\mcN)$ and put $\mbK=(\mbK_n,\mu_n)_{n\in\mbbN}$. Clearly $\mu_n$ is a well defined probability measure due to $Q$ implicitly being defined by the labeling of each structure. Since $\xi$ is a $V-$formula, it is almost surely true in $\mbK$ that $\xi$ is a restricted equivalence relation and the $(\Delta,\xi)-$extension properties are almost surely satisfied. By Corollary \ref{simpcor}, $Th(\mcM)$ is axiomatized by the $(\Delta,\xi)-$extension properties and thus $T_\mbK = Th(\mcM)$.
\end{proof}
In previous works on the subject of finding $0-1$ laws the previous theorem and Proposition \ref{lawprop} are not standard since the sets considered are not (almost surely) the set of substructures of a model of the almost sure theory under the uniform measure. The question thus arise if we may create the same results in that context, why we now turn to studying what we call random structures. 
\begin{defi}\label{rngstrdef}
Let $\mcM$ be a structure and $\mbK_n = \{\mcA: A = [n] , \exists f: \mcA \rightarrow \mcM \text{ embedding}\}$ with a probability measure $\mu_n$ such that for $\mcN\in\mbK_n$, $\mu_n(\mcN)= 1/|\mbK_n|$. We say that $\mcM$ is a \textbf{random structure} if $\mbK=(\mbK_n,\mu_n)_{n\in\mbbN}$ has a $0-1$ law and $\mcM\models T_\mbK.$
\end{defi}
Lemma \ref{cllemma} and Corollary \ref{vectcor} imply that an $\omega-$categorical structure $\mcM$, where $(\mcM,acl_\mcM)$ forms a vector space pregeometry, is not a random structure. In our previous examples, by definition, the infinite structures in \ref{omegaex} and \ref{partexa} are both random structures. However Example \ref{nonrigidexa} does not give a random structure, which follows from the previous proposition as it is not $\omega-$stable. Theorem \ref{reductprp} does however show that all the structures considered in this article are at least close to being random structures, which we will now prove.
\begin{proof}[Proof of Theorem \ref{reductprp}]
Assume that $tp(x/acl_{\mcM^{eq}}(\emptyset)) = tp(y/acl_{\mcM^{eq}}(\emptyset))$ has exactly $l$ equivalence classes. Enumerate the classes and let $\Delta=\{\delta_{i,j}\}_{i,j\in[l]}$ be such that $\delta_{i,j}$ is the set of all possible binary atomic diagrams between class $i$ and class $j$. Put $r = \max_{i,j} |\delta_{i,j}|+1$ and let $V'=V\cup\{Q\} \cup \{R^1_{i,j},\ldots,R^{r-|\delta_{i,j}|}_{i,j}\}_{i\leq j}$ where $Q$ and $R_{i,j}^t$ are binary relational symbols not in $V$. For each $i\leq j$ and for some $P(x,y)\in \delta_{i,j}$ we will define new distinct atomic diagrams $P_{0},\ldots,P_{r-|\delta_{i,j}|}$ inductively. Let $P_0$ be the atomic diagram $P(x,y)$ and if $P_{t}(x,y)$ is defined then let $P_{t+1}$ be the atomic diagram $P_{t}(x,y)\wedge R_{i,j}^t(x,y) \wedge R_{i,j}^t(y,x)$. Let $\delta_{i,j}' = \delta_{i,j} \cup \{P_1,\ldots,P_{r-|\delta_{i,j}|}\}$ for $i\leq j$ but if $i>j$ let $\delta_{i,j}'$ be the set of reversed atomic diagrams in $\delta_{j,i}'$. Further add $Q(x,y) \wedge Q(x,x)$ to each atomic diagram in $\delta'_{i,j}$ if and only if $i=j$. 
\\\indent 
It is now clear that $|\delta_{i,j}'|= |\delta_{i',j'}'|$ for each $i,j,i',j'\in[l]$ and $\Delta'=\{\delta'_{i,j}\}_{i,j\in[l]}$ is a $(l,0,Q)-$compatible set. Proposition \ref{lawprop} now implies that there is a countable structure $\mcM'$ which satisfies $(\Delta', Q)-$extension properties, and thus $\mcM'\reduct V$ satisfies $(\Delta,\xi)-$extension properties which in turn implies that $\mcM\cong\mcM'\reduct V$. 
Let $\mbK_n = \{\mcA : A = [n], \exists f: \mcA \rightarrow \mcM' \text{ embedding}\}$ under the uniform measure $\mu_n$. 
As there are an equal amount of possible atomic diagrams between and inside $Q-$equivalence classes and each equivalence class is distinguished by some unique relational symbol it follows quickly that almost surely $\mcA\in \mbK_n$ will contain $l$ $Q-$equivalence classes with more than $\log (n)$ elements in each class. It is now straight forward to show that $\mcM'$ is a random structure in the same way as we showed the $0-1$ law of Proposition \ref{lawprop}.
\end{proof}
The following example describes a structure which satisfies all the assumption of Proposition \ref{reductprp} but is not a random structure. We may thus conclude that being a reduct of a random structure is the best we can get in general for such structures.
\begin{exa}\label{necessaryreduct}
Let V be the vocabulary $\{E_1,E_2,P\}$ where $E_1,E_2$ are binary and $P$ is unary. Let $\mcM$ be the $V-$structure consisting of the disjoint union of the structures $\mcG_1$ and $\mcG_2$ such that the relation $P$ holds for all elements in $\mcG_2$. The  countable structures $\mcG_1$ and $\mcG_2$ are models of the almost sure theory of the class consisting of all finite structures with two respectively one symmetric anti-reflexive relation under the uniform measure (hence $\mcG_2$ is the random graph). It is a quick exercise (which may use Theorem \ref{simpthm}) to show that $\mcM$ is $\omega-$categorical, simple with $SU-$rank 1 and $acl_\mcM(\emptyset)=\emptyset$. Let 
\[\mbK_n = \{\mcA : A = [n], \mcA \hookrightarrow \mcM\}\]
and let 
$\mbC_n = \{\mcN\in\mbK_n : \exists f:\mcN\rightarrow \mcG_1 \text{ embedding}\}$. Note that $|\mbC_n| = 4^{\binom{n}{2}}$. We may then calculate the proportion of structures in $\mbK_n$ which belong to $\mbC_n$:
\[\frac{|\mbC_n|}{|\mbK_n|} = 1 - \frac{\sum_{i=1}^n\binom{n}{i}|C_{n-i}|2^{\binom{i}{2}}}{|\mbK_n|} \geq 1 - \frac{\sum_{i=1}^{n}\binom{n}{i} 4^{\binom{n}{2}}2^{\binom{i}{2}}}{4^{\binom{n}{2}}} =  \]
\[1- \sum_{i=1}^n \binom{n}{i}4^{(-2in + i^2+i)/2} 2^{(i^2-i)/2} \geq 1- \sum_{i=1}^n \binom{n}{i}4^{(-in)/2} 2^{(in)/2} \geq \]
\[ 1- \sum_{i=1}^n \binom{n}{i} \left(\sqrt{\frac{1}{2}}\right)^{in}  \geq 1- \frac{1}{n}\sum_{i=1}^n\binom{n}{i}n^{-i}  = 1- \frac{1}{n} \big((1+ 1/n)^n-1\big )\]
which tends to 1 as $n$ tends to infinity. Thus almost surely $\mbK_n$ equipped with the uniform measure will only contain substructures of $\mcG_1$. Hence $\mcM\not\models T_\mbK$ because the theories $Th(\mcM)$ and $T_\mbK$ satisfy different extension axioms. We conclude that $\mcM$ is not a random structure.
\end{exa}
The previous example is a quite small and easy case. If we would have multiple equivalence classes and multiple atomic diagrams between them then the calculations would be considerably harder. We leave for future combinatoric research to deduce exactly which of the structures of Proposition \ref{reductprp} are random structures and which are just reducts of such. \\\indent
The sets $\mbK_n$ in Proposition \ref{lawprop} are constructed in a specific way, taking care that all $l$ equivalence classes are nonempty and express all the possible atomic diagrams of the spanning formula. The reason for taking such caution, and the reason why we assumed $acl(\emptyset)=\emptyset$ in Theorem \ref{reductprp}, is that $acl_\mcM(\emptyset)$ will almost surely disappear from the set of embeddable structures, unless $\mcM$ is $\omega-$stable. 
\begin{prp}\label{stablerngstr}Let $\mcM$ be a structure which is $\omega-$categorical, simple, with $SU-$rank 1 and with trivial pregeometry such that $acl_\mcM(\emptyset)\neq \emptyset$. If $\mcM$ is a random structure then $\mcM$ is $\omega-$stable. 
\end{prp}
\begin{proof}
If $\mcM$ is not $\omega-$stable, then by Theorem \ref{simpthm} and Corollary \ref{stablethm} there exists a restricted equivalence relation $\xi$ with $l$ equivalence classes such that between some two equivalence classes, or inside one equivalence class, there are multiple possible atomic diagrams. We assume that the equivalence class $B$ contains multiple atomic diagrams inside of it. The calculations are similar (but slightly more technical) in the second case. Let $a\in acl(\emptyset)$ and assume that $(\mbK_n,\mu_n)_{n\in\mbbN}$ is as in the definition of a random structure. Since $\mcM$ is a random structure $\xi$ will almost surely define an equivalence relation where the atomic diagrams between some class with one element and the other classes are the same as $a$ has to the other classes, while one class will almost surely contain the same atomic diagrams as $B$.\\\indent 
The atomic diagram between any equivalence class and a base set is uniquely determined. As $\mcM$ is a random structure, almost surely in $\mbK_n$ there will be more than $f(n)$ elements in the $\xi-$equivalence class with the same atomic diagrams as in $B$, for some increasing function $f$ and only one element in the $\xi-$equivalence class approximating $a$. For each $\mcA\in \mbK_n$ which contain an element approximating $a$ there almost surely exists $2^{f(n)}$ structures not approximating $a$. Thus almost surely $\mcA\in\mbK_n$ will not contain an equivalence class which correspond to $a$. This is a contradiction.
\end{proof} 
\begin{rmk}
Definition \ref{rngstrdef} could have been made in another way to include a wider range of structures. If we had instead done the definition considering $\mbK_n= \{\mcA : A=[n], \exists f: \mcA\rightarrow \mcM \text{ embedding and } acl_\mcM(\emptyset) \subseteq \im(f) \}$ then Proposition \ref{stablerngstr} would no longer hold, Theorem \ref{reductprp} would be possible to generalize and Example \ref{nonrigidexa} would give a random structure. Do however note that Example \ref{necessaryreduct} would still be viable, and thus we may not even in this setting skip the reduct part of Theorem \ref{reductprp}.
\end{rmk}
Just as we in Section \ref{smsection} got corollaries regarding the strongly minimal theories we get, in this section, the following corollary which sums up how the strongly minimal $\omega-$categorical structures may be approximated.

\begin{cor}\label{smlastcor}Let $V$ be any finite relational vocabulary. Assume that $T$ is $\omega-$categorical and strongly minimal. For countable $\mcM\models T$ the following are equivalent:
\begin{itemize}
\item[(i)] $\mcM$ consists of an indiscernible set over a finite tuple.
\item[(ii)] The algebraic closure in $\mcM$ is trivial.
\item[(iii)] $\mcM$ is a random structure.
\end{itemize}
\end{cor}

\begin{proof}
(ii) implies (i) follows since in strongly minimal theories if $\mcM\models T$ and $\{a_1,\ldots,a_n\},\{b_1,\ldots,b_n\}\subseteq M$ are independent sets then $tp(a_1,\ldots,a_n)=tp(b_1,\ldots,b_n)$. (iii) implies (ii) follows from Lemma \ref{cllemma} in the same way as we used it in proving Proposition \ref{smlma}. (i) implies (iii) may be proved by using a generalization of Theorem \ref{rngstrthm} to the context of finite relational vocabularies in the context of $\omega-$stable structures in accordance to Remark \ref{highervocrmk}. By building the sets of structures like in Proposition \ref{lawprop} and the proof of Theorem \ref{rngstrthm} we may conclude that $|N|=n$ for $\mcN\in\mbK_n$.
\end{proof}

\textbf{Acknowledgement.} The author would like to thank Professor Vera Koponen of Uppsala University for valuable discussion and remarks.


\begin{thebibliography}{9}
\bibitem{AK2} O. Ahlman, V. Koponen, \textit{Limit laws and automorphism groups of random non-rigid structures}, Journal of Logic \& Analysis 7:2 (2015) 1-53.
\bibitem{AK3} O. Ahlman, V. Koponen, \textit{On sets with rank one in simple homogeneous structures}, Fundamenta Mathematicae, Vol. 228 (2015) 223-250.
\bibitem{AK} O. Ahlman, V. Koponen, \textit{Random l-colourable structures with a pregeometry}, arXiv 1207.4936v1.
\bibitem{AFP} N. Ackerman, C. Freer, R. Patel, \textit{Invariant meassures concentrated on countable structures}, arxiv 1206.4011v3.
\bibitem{B} J.T. Baldwin, \textit{The mathematics of random graphs}, Annals of Pure and Applied Logic 143 (2006) 20-28.
\bibitem{C} K.J. Compton, \textit{The computational complexity of asymptotic problems I: partial orders}, Inform. and comput. 78 (1988), 108-123.
\bibitem{EF} H-D. Ebbinghaus, J. Flum, \textit{Finite model theory}, Springer verlag (2000).
\bibitem{EKR} P. Erd\H{o}s, D. Kleitman, B. Rothschild, \textit{Asymptotic Enumeration of $K_n-$free graphs}, Colloquio Internazionale sulle Teorie Combinatorie (Rome, 1973), Atti dei Convegni Lincei , Vol. 17, Accad. Naz. Lincei, Rome (1976), 19-27. 
{\em Acta Math. Acad. Sci. Hungar.}, Vol. 14 (1963) 295--315.
\bibitem{F} R. Fagin, \textit{Probabilities on finite models}, J. Symbolic Logic 41(1976), no. 1, 55-58.
\bibitem{Fr} Fra\"{i}ss\'{e}, \textit{Sur certaines relations qui g\'en\'eralisent lorder des nombres rationnels}, C. R. Acad. Sci. Paris 237 (1953) 540-542.
\bibitem{H} W. Hodges, \textit{Model theory}, Camebridge University Press (1993).
\bibitem{KR} D.J. Kleitman, B.L. Rothschild, \textit{Asymptotic enumeration of partial orders on a finite set}, Trans. Amer. Math. Soc. 205 (1975) 205-220.
\bibitem{KPR} P.G. Kolaitis, H.J. Prömel, B.L. Rothschild, \textit{$K_{l+1}$-free graphs: asymptotic structure and a 0-1 law}, Trans. Amer. Math. Soc. 303 (1987) 637-671.
\bibitem{K} V. Koponen, \textit{Asymptotic probabilities of extension properties and random l-colourable structures}, Annals of pure and applied logic, Vol 163 (2012) 391-438.
\bibitem{MT} D. Mubayi, C. Terry, \textit{Discrete metric spaces: structure, enumeration, and 0-1 laws}, arXiv: 1502.01212 (2015).
\bibitem{S} J. Spencer, \textit{The strange logic of random graphs}, Springer-Verlag (2000).
\bibitem{W} F. O. Wagner, {\em Simple Theories}, Kluwer Academic Publishers (2000).
\end{thebibliography}
\end{document}